\documentclass[a4paper,11pt]{amsart}

\providecommand{\bysame}{\leavevmode\hbox to3em{\hrulefill}\thinspace}
\providecommand{\MR}{\relax\ifhmode\unskip\space\fi MR }

\providecommand{\href}[2]{#2}

\usepackage{enumerate, booktabs, mathrsfs}
\usepackage{amsmath, amsfonts, amssymb, amsthm}
\usepackage[all]{xy}
\usepackage[height=22.5cm, width=15.5cm]{geometry}
\usepackage[active]{srcltx}

\numberwithin{equation}{section}

\theoremstyle{plain}
\newtheorem{thm}{Theorem}[section]
\newtheorem{prop}[thm]{Proposition}
\newtheorem{lem}[thm]{Lemma}
\newtheorem{cor}[thm]{Corollary} 
\newtheorem*{thm*}{Theorem}
\newtheorem*{thm1}{Main Theorem}
\newtheorem*{thm2}{Theorem 3.1}
\newtheorem*{thm3}{Theorem 5.1}

\theoremstyle{definition}
\newtheorem{defn}[thm]{Definition}

\theoremstyle{remark}
\newtheorem{rem}[thm]{Remark}
\newtheorem{ex}[thm]{Example}

\newcommand{\mbb}[1]{\mathbb{#1}}
\newcommand{\wt}[1]{\widetilde{#1}}
\newcommand{\wh}[1]{\widehat{#1}}
\newcommand{\ol}[1]{\overline{#1}}
\newcommand{\lie}[1]{{\mathfrak{#1}}}
\newcommand{\abs}[1]{\lvert #1\rvert}
\newcommand{\norm}[1]{\lVert #1\rVert}

\DeclareMathOperator{\Ad}{Ad}

\DeclareMathOperator{\eps}{\varepsilon}

\DeclareMathOperator{\g}{\mathfrak g}

\DeclareMathOperator{\G}{\Gamma}
\DeclareMathOperator{\La}{\Lambda}

\title{Pseudoconvex domains spread over complex homogeneous manifolds}

\author{Bruce Gilligan}
\address{Department of Mathematics and Statistics, University of Regina,
Regina, Canada S4S 0A2}  
\email{gilligan@math.uregina.ca}  

\author{Christian Miebach}
\address{Laboratoire de Math\'ematiques Pures et Appliqu\'ees, CNRS-FR~2956,
Universit\'e du Littoral, 50, rue F.~Buisson, 62228 Calais Cedex, France}
\email{miebach@lmpa.univ-littoral.fr}

\author{Karl Oeljeklaus}
\address{UFR-MIM et LATP (UMR-CNRS 6632), Universit\'e d'Aix-Marseille, 39, rue
F.~Joliot-Curie, 13453 Marseille  Cedex 13, France}
\email{karloelj@cmi.univ-mrs.fr}

\date{April 5, 2012} 


\thanks{The authors would like to thank Peter Heinzner and the SFB/TR 12 for an
invitation to the Ruhr-Universit\"at Bochum where this project was started. Our
work was also partially supported by an NSERC Discovery Grant.}

\begin{document}

\begin{abstract}
Using the concept of inner integral curves defined by Hirschowitz we generalize
a recent result by Kim, Levenberg and Yamaguchi concerning the obstruction of
a pseudoconvex domain spread over a complex homogeneous manifold to be Stein.
This is then applied to study the holomorphic reduction of pseudoconvex complex
homogeneous manifolds $X=G/H$. Under the assumption that $G$ is solvable or
reductive we prove that $X$ is the total space of a $G$-equivariant holomorphic
fiber bundle over a Stein manifold such that all holomorphic functions on the
fiber are constant.
\end{abstract}

\maketitle   

\section{Introduction}

Let $G$ be a connected complex Lie group and $H$ a closed complex subgroup of
$G$. The complex homogeneous manifold $X=G/H$ admits a Lie theoretic holomorphic
reduction $\pi\colon G/H\to G/J$ where $G/J$ is holomorphically separable and
$\mathscr{O}(G/H)\simeq\pi^*\mathscr{O}(G/J)$. In general, the base $G/J$ is
not Stein nor do we have $\mathscr{O}(J/H)=\mbb{C}$. In some cases one can say
more. If $G$ is solvable, then $G/J$ is always a Stein manifold, although
$\mathscr{O}(J/H)=\mbb{C}$ is not true in general (see~\cite{HuOe86}). If $G$ is
complex reductive, then due to~\cite{BO} there is a factorization $G/H\to
G/\ol{H}\to G/J$ where $\ol{H}$ denotes the Zariski closure of $H$ in $G$;
moreover, $G/J$ is a quasi-affine variety. In general,
$\mathscr{O}(\ol{H}/H)\not=\mbb{C}$ and $\ol{H}/H$ can even be Stein.

As the main result of this paper we prove the following theorem about the
holomorphic reduction of \emph{pseudoconvex} complex homogeneous manifolds
$X=G/H$ where $G$ is solvable or reductive.

\begin{thm1}
Suppose that the complex homogeneous manifold $X=G/H$ is pseudoconvex and let
$X=G/H\to G/J$ be its holomorphic reduction.
\begin{enumerate}[(1)]
\item If $G$ is a complex reductive Lie group, then the base $G/J$ is Stein and 
the fiber $J/H$  satisfies $\mathscr{O}(J/H)=\mbb{C}$.   
If $X$ is K\"{a}hler as well, then $J/H$ is a product of the Cousin group $\ol{H}/H$ with the 
 homogeneous rational manifold $J/\ol{H}$.     
\item If $G$ is solvable, then the fiber $J/H$ is a Cousin group tower and thus
$\mathscr{O}(J/H)=\mbb{C}$.
\end{enumerate}
\end{thm1}

An open question is whether the holomorphic function algebra 
$\mathscr{O}(G/H)$ is a Stein algebra for $G/H$ pseudoconvex and $G$ general.     
This lies beyond the scope of the present paper and is not addressed here.   

As a first step towards the proof of this theorem we discuss the Levi problem
for pseudoconvex domains spread over complex homogeneous spaces and present a
Lie theoretic description of the obstruction to their being Stein. A
characterization of relatively compact, smoothly bounded, pseudoconvex domains 
$D$ in complex homogeneous manifolds such that $D$ is not Stein is given
in~\cite{KLY}. The incorporation of  methods of Hirschowitz~\cite{Hir75} allows
us to simplify their proof and to show that the assumptions on the smoothness of
the boundary and relative compactness of $D$ are not needed. One of the
essential tools that Hirschowitz uses is the concept of an {\it inner integral
curve}, i.e., a non-constant holomorphic image of $\mathbb C$ in $D$ that is
relatively compact and is the integral curve of a vector field. Indeed,
Hirschowitz proves that if a non-compact pseudoconvex domain $D$ is spread over
an infinitesimally homogeneous complex manifold and $D$ has no inner integral
curves, then $D$ is Stein.

Our generalization of the main result of~\cite{KLY} reads then as follows.

\begin{thm2}
Let $p\colon D\to X$ be a pseudoconvex domain spread over the complex
homogeneous manifold $X=G/H$ such that $p(D)$ contains the base point $eH\in X$.
If $D$ is not Stein, then there exist a connected complex Lie subgroup $\wh{H}$
of $G$ with $H^0\subset\wh{H}$ and $\dim H<\dim\wh{H}$ and a foliation
$\mathscr{F}=\{F_x\}_{x\in D}$ of $D$ such that
\begin{enumerate}[(1)]   
\item every leaf of $\mathscr{F}$ is a relatively compact immersed complex
submanifold of $D$,
\item every inner integral curve in $D$ passing through $x\in D$ lies in the
leaf $F_x$ containing $x$, and
\item the leaves of $\mathscr{F}$ are homogeneous under a covering group of
$\wh{H}$.
\end{enumerate}
\end{thm2}

In Proposition~\ref{Prop:normalizing} we show furthermore the existence of an
open subgroup $H^*$ of $H$ such that $D$ can be realized as a domain spread over
$G/H^*$ and that $H^*$ normalizes $\wh{H}$. This will be important in the proof
of our Main Theorem.

Moreover, we have the following strengthening of this theorem in the projective
setting.

\begin{thm3}
Suppose $G$ is a connected complex Lie group acting holomorphically on 
$\mbb{P}_n(\mbb{C})$ and $X=G/H$ is an orbit. Then, every pseudoconvex domain
spread over $X$ is holomorphically convex and the fibers of its Remmert
reduction are rational homogeneous manifolds.
\end{thm3}

Let us briefly outline the organization of this paper. In Section~2 we summarize
Hirschowitz' results in a form suitable for our needs. In Section~3 we prove
Theorem~\ref{Thm:Main} and discuss several applications of it. In the fourth
section we investigate when the foliation of a pseudoconvex domain has compact
leaves, which leads in Section~5 to the proof of Theorem~\ref{projcase}. The
sixth section contains a generalization of Kiselman's minimum principle that is
used in the last two sections in order to prove our Main Theorem in the reductive
and solvable case, respectively.

Throughout this paper we will denote Lie groups by upper case letters and their
Lie algebras by the corresponding fracture letters.
 
\section{A result of Hirschowitz}

Following Hirschowitz~\cite{Hir74} and~\cite{Hir75} we call a complex manifold
$X$ {\emph{infinitesimally homogeneous}} if every tangent space of $X$ is
generated by global holomorphic vector fields. Every complex manifold
homogeneous under the action of a Lie group of holomorphic transformations
is infinitesimally homogeneous, and more generally every domain spread over such
a manifold has this property. Here we understand by a \emph{domain spread over
$X$} a pair $(D,p)$ where $D$ is a connected complex manifold and $p\colon D\to
X$ is locally biholomorphic.   

Throughout the rest of the paper we fix a complex homogeneous space $X=G/H$
where $G$ is a connected complex Lie group and $H$ is a (not necessarily
connected) closed complex subgroup of $G$. We view an element $\xi\in\lie{g}$ as
a right invariant vector field on $G$ which we push down to a holomorphic vector
field $\xi_X$ on $X=G/H$. If $p\colon D\to X$ is a domain spread over $X$, we
denote by $\wt{\xi}_X$ the lift of $\xi_X$ to $D$. An \emph{inner integral curve
in $D$} is a non-constant holomorphic map $\mbb{C}\to D$ with relatively compact
image in $D$ which is the integral curve of some vector field $\wt{\xi}_X$ with
$\xi\in\lie{g}$.

In this paper a complex manifold is called {\emph{pseudoconvex}} if it admits
a continuous plurisubharmonic exhaustion function. Note that every
holomorphically convex complex manifold is pseudoconvex. We will use the
following special case of~\cite[Theorem~4.1]{Hir75}.

\begin{thm}\label{Thm:Hirsch}  
Let $p\colon D\to X$ be a non-compact pseudoconvex domain spread over a complex
homogeneous manifold $X=G/H$.
\begin{enumerate}[(1)]
\item If $D$ is not Stein, then $D$ contains an inner integral curve.
\item If $X$ is a compact rational variety, then $D$ is holomorphically convex.
\item If $X$ is an irreducible compact rational variety, then $D$ is Stein.
\end{enumerate}
\end{thm}

\begin{rem}
Let $X$ be infinitesimally homogeneous. If $X$ is compact rational, then
$X=G/P$ where $G$ is complex semisimple and $P$ is a parabolic subgroup of $G$.
Such $X$ is irreducible if and only if $G$ is simple and $P$ is maximal
parabolic.
\end{rem}

The following example shows that Hirschowitz' theorem does not hold for locally
Stein domains in complex homogeneous manifolds.

\begin{ex}
Let $D$ be the punctured unit ball in $\mbb{C}^2\setminus\{0\}\simeq G/H$ where
$G={\rm{SL}}(2,\mbb{C})$ and $H=\left(\begin{smallmatrix}1&\mbb{C}\\0&1
\end{smallmatrix}\right)$. Since $\partial D$ is strictly Levi-convex, $D$ is
locally Stein in $G/H$, but not pseudoconvex (in which case $D$ would be Stein).
Note that $D$ does not contain an inner integral curve.
\end{ex}

The following result (see~\cite[Theorem~2.1]{Hir74}) implies that such an
example does not exist when the domain $p(D)$ is relatively compact in $X$ and
$p$ has finite fibers.

\begin{thm}
Let $p\colon D\to X$ be a locally Stein domain spread over $X=G/H$. If $p(D)$ is
relatively compact and $p$ has finite fibers, then $D$ is pseudoconvex.
\end{thm}

\begin{rem}\label{Serre}   
As a direct application of Theorem~\ref{Thm:Hirsch} we obtain the following:
\emph{If each holomorphic map $\mbb{C}\to X$ with relatively compact image is
constant, then every pseudoconvex domain spread over $X$ is Stein.} This
observation applies e.g.~when $X$ is Brody-hyperbolic or if there exists a
holomorphic map $f\colon X\to Y$ such that $Y$ and all fibers of $f$ are
holomorphically separable. 
\end{rem}     

The following example comes from a construction in~\cite{CL} and illustrates the
second case in Remark~\ref{Serre}.

\begin{ex}\label{Ex:CoeureLoeb}
Let $A=\left(\begin{smallmatrix}2&1\\1&1\end{smallmatrix}\right)$ and take $D:=
\log(A)$ the unique real logarithm of $A$. For $K=\mbb{Z},\mbb{R},\mbb{C}$ we
define $G_K$ to be the semi-direct product $K\ltimes K^2$ with group law
\begin{equation*}
(x_1,y_1)\cdot(x_2,y_2):=\bigl(x_1+x_2,\exp(x_1D)y_2+y_1\bigr).
\end{equation*}
Then $G_\mbb{R}/G_\mbb{Z}$ is a compact totally real submanifold of $G_\mbb{C}/
G_\mbb{Z}$ and the map $p\colon G_\mbb{C}/G_\mbb{Z}\to G_\mbb{C}/(G_\mbb{C}'
G_\mbb{Z})\simeq\mbb{C}^*$ is a holomorphic fiber bundle with fiber $\mbb{C}^*
\times\mbb{C}^*$. It is known that $\mathscr{O}(G_\mbb{C}/G_\mbb{Z})=
p^*\mathscr{O}(\mbb{C}^*)$, in particular $G_\mbb{C}/G_\mbb{Z}$ is not Stein.
However, since base and fiber are Stein, $G_\mbb{C}/G_\mbb{Z}$ cannot contain an
inner integral curve. Therefore every pseudoconvex domain spread over
$G_\mbb{C}/G_\mbb{Z}$ is Stein by Theorem~\ref{Thm:Hirsch}.

We would like to point out that $G_\mbb{C}/G_\mbb{Z}$ contains pseudoconvex
domains of the form $D=p^{-1}(U)$ for a domain $U\subset\mbb{C}^*$ which are
non-trivial in the sense that the restricted fiber bundle $p|_D$ is non-trivial.
In fact, Dan Zaffran proved in~\cite{Zaf} that for an annulus $U \subset
\mbb{C}^*$ of modulus smaller than a certain constant, the inverse image
$p^{-1}(U) \subset G_\mbb{C}/G_\mbb{Z}$ is Stein, so in particular it is
pseudoconvex. Of course every domain of the form $p^{-1}(U)$ is locally Stein in
$G_\mbb{C}/G_\mbb{Z}$, but in general not pseudoconvex.
\end{ex}

\section{Existence of the foliation}   

In this section we consider a non-Stein pseudoconvex domain $p\colon D\to X$
spread over the complex homogeneous manifold $X=G/H$.
Generalizing~\cite[Theorem~6.4]{KLY} we will show that there exists a connected
complex Lie subgroup $\wh{H}$ of $G$ that induces a holomorphic foliation of
$D$ having relatively compact leaves such that    
every leaf of this foliation contains all of the inner integral curves of $D$ 
passing through any of its points. Note that it is not assumed that $p(D)$ is
relatively compact or has smooth boundary. What happens when these leaves are
closed, and thus compact, is addressed in Section~\ref{Section:holconv}.

\begin{thm}\label{Thm:Main}   
Let $p\colon D\to X$ be a pseudoconvex domain spread over $X=G/H$ such that
$p(D)$ contains the base point $eH\in X$. If $D$ is not Stein, then there exist
a (not necessarily closed) connected complex Lie subgroup $\wh{H}$ of $G$ with
$H^0\subset\wh{H}$ and $\dim H<\dim\wh{H}$ as well as a foliation
$\mathscr{F}=\{F_x\}_{x\in D}$ of $D$ such that
\begin{enumerate}[(1)]   
\item every leaf of $\mathscr{F}$ is a relatively compact immersed complex
submanifold of $D$,
\item every inner integral curve in $D$ passing through $x\in D$ lies in the
leaf $F_x$ containing $x$, and
\item the leaves of $\mathscr{F}$ are homogeneous under a covering group of
$\wh{H}$ and the restriction of $p$ to a leaf in $D$ is a finite covering map
onto its image in $X$.
\end{enumerate}
\end{thm}

We will prove Theorem~\ref{Thm:Main} in several steps and start by defining the
group $\wh{H}$. For this let $p\colon D\to X$ be a domain over $X=G/H$ and
choose $x_0\in D$ such that $p(x_0)=eH$. We will denote the local flow of the
holomorphic vector field $\wt{\xi}_X$ by $(t,x)\mapsto e^{t\xi} \cdot x$ and its
maximal complex integral manifold through $x$ by $F^\xi_x$. Note that $t$ is a
\emph{complex} parameter and that $F^\xi_x$ is an immersed complex curve. Then
we set
\begin{equation}\label{Eqn:Hhat}
\wh{\lie{h}}:=\bigl\{\xi\in\lie{g};\ \wt{\xi}_X\varphi(x_0)=0\text{ for every
continuous plurisubharmonic function $\varphi$ on $D$}\bigr\}
\end{equation}
where the derivative $\wt{\xi}_X\varphi(x_0)=\left.\frac{d}{dt}\right|_0
\varphi(e^{t\xi}\cdot x_0)$ of a continuous function has to be understood in the
distributional sense and is taken with respect to the complex parameter $t$. If
$X=G/H$ is unimodular, i.e., if $X$ has a $G$-invariant Borel measure, and if
$D\subset X$, then we can use the usual convolution technique in order to
approximate continuous plurisubharmonic functions on $D$ by smooth ones. In this
case we may replace ``continuous" by ``smooth" in the definition of
$\wh{\lie{h}}$. We shall see later that $\wh{\lie{h}}=\lie h$ if and only if a
pseudoconvex $D$ is Stein.

\begin{ex}
We present here an example which turns out to be important in the rest of the
paper. Let $G$ be a {\it Cousin group}, i.e., a connected complex Lie group
without non-constant holomorphic functions. It is well-known that $G$ is a
quotient of $ (\mbb C^{n},+)$ by a discrete subgroup $\Gamma$  of rank $n+m$,
$1\leq m\leq n$, generating $\mbb C^{n}$ over $\mbb C$. With $V :=\langle \Gamma
\rangle_{\mbb R}$ and  $W:= V \cap iV$, one has furthermore that $V/\Gamma $ is
the maximal compact subgroup of $G$ and that $W + \Gamma$ is dense in $V$.    
Hence we have $\wh{\lie h}=W$ under the identification $\lie{g}\simeq\mbb
C^{n}$.
\end{ex}

Since we can lift the vector field $\xi_X$ from $X$ to $D$ for every
$\xi\in\lie{g}$, we obtain a \emph{local} holomorphic action of $G$ on $D$ such
that $p\colon D\to X$ is equivariant. This means that there exist an open
neighborhood $\Omega\subset G\times D$ of $\{e\}\times D$ such that $\bigl\{g\in
G;\ (g,x)\in\Omega\bigr\}$ is connected for every $x\in D$, as well as a
holomorphic map $\Phi\colon\Omega\to D$, $\Phi(g,x)=:g\cdot x$, fulfilling the
usual axioms of a group action. For more details we refer the reader
to~\cite{HeIan}. Note that the local $G$-action on $D$ is in general not
globalizable unless $p\colon D\to X$ is schlicht. For this reason there will in
general be no maximal domain of definition $\Omega$ of the local $G$-action. 

For the readers' convenience we repeat some arguments from~\cite{Hir75}
and~\cite{KLY} in order to give the proof of the following.    

\begin{lem}\label{Lem:Subalgebra}
Let $p\colon D\to X$ be pseudoconvex. Then the set $\wh{\lie{h}}$ defined in
equation~\eqref{Eqn:Hhat} is a complex Lie subalgebra of $\lie{g}$.
\end{lem}

\begin{proof}
The key point is to observe that if $\wt{\xi}_X\varphi(x_0)=0$ for every
continuous plurisubharmonic function on $D$, then every such $\varphi$ must be
constant on $F^\xi_{x_0}$. To see this, let $\rho$ be a continuous
plurisubharmonic exhaustion function of $D$ and set $D_\alpha:=\bigl\{x\in D;\
\rho(x)<\alpha\bigr\}$. Choose $\alpha\in\mbb{R}$ such that $x_0\in D_\alpha$.
For $\abs{t}$ sufficiently small we have
\begin{equation*}
e^{t\xi}\cdot D_{\alpha+1}\supset D_\alpha\ni x_t:= e^{t\xi}\cdot x_0
\end{equation*}
Hence, we have the holomorphic map $e^{t\xi}\colon D_\alpha\to D_{\alpha+1}$
and therefore $\varphi_{-t}:=\varphi\circ e^{t\xi}$ is continuous
plurisubharmonic on $D_\alpha$ whenever $\varphi$ is continuous plurisubharmonic
on $D_{\alpha+1}$. Following~\cite[Proposition~1.6]{Hir75} we construct a
continuous plurisubharmonic function $\psi_{-t}$ on $D$ which coincides with
$\varphi_{-t}$ in a neighborhood of $x_0$. Choose $\beta\in\mbb{R}$ such that
$\varphi_{-t}(x_0)<\beta<\alpha$ and note that $K:=\rho^{-1}(\beta)\subset
D_\alpha$ is compact. Then choose a convex increasing function $\chi$ on
$\mbb{R}$ fulfilling
\begin{align*}
\chi\bigl(\rho(x_0)\bigr) &< \varphi_{-t}(x_0)\quad\text{ and}\\
\chi(\beta) &> \norm{\varphi_{-t}}_K.
\end{align*}
Finally, define $\psi_{-t}\colon D\to\mbb{R}$ by
\begin{equation*}
\psi_{-t}(x):=
\begin{cases}
\max\bigl(\varphi_{-t}(x), \chi\circ\rho(x)\bigr) &: \rho(x)\leq\beta\\
\chi\circ\rho(x) &: \rho(x)\geq\beta.
\end{cases}
\end{equation*}
One checks directly that $\psi_{-t}$ is continuous plurisubharmonic and
coincides with $\varphi_{-t}$ in some neighborhood of $x_0$. Consequently, we
may calculate
\begin{equation*}
\wt{\xi}_X\varphi(x_{-t})=\left.\frac{d}{ds}\right|_t\varphi\bigl(e^{-s\xi}
\cdot x_0\bigr)=\wt{\xi}_X\varphi_{-t}(x_0)=\wt{\xi}_X\psi_{-t}(x_0)=0.
\end{equation*}
Hence, the set of $x_t$ such that $\wt{\xi}_X\varphi(x_t)=0$ holds for every
continuous plurisubharmonic $\varphi$ is open in $F^\xi_{x_0}$, and since it is
also closed, we see that $\wt{\xi}_X\varphi$ vanishes on $F^\xi_{x_0}$ as a
distribution for every continuous plurisubharmonic $\varphi$. But this implies
in turn that $\varphi$ is constant on $F^\xi_{x_0}$. Thus the proof of our first
claim is finished.

Applying this fact to the continuous exhaustion function $\rho$ of $D$, we
conclude that for every $\xi\in\wh{\lie{h}}$ the maximal integral manifold
$F^\xi_{x_0}$ is relatively compact and contained in $D$. In particular, we have
$F^\xi_{x_0}=\mbb{C}^\xi\cdot x_0$ where $\mbb{C}^\xi$ is the universal covering
of $\exp(\mbb{C}\xi)$. In order to finish the proof that $\wh{\lie{h}}$ is a
complex Lie subalgebra of $\lie{g}$, we will show that for every finite
collection $\xi_1,\dotsc,\xi_k\in\wh{\lie{h}}$ the map
\begin{equation*}
(t_1,\dotsc,t_k)\mapsto\rho\bigl(e^{t_{1}\xi_{1}}\dotsb e^{t_{k}\xi_{k}}\cdot x_0
\bigr)
\end{equation*}
defined on a neighborhood of $0\in\mbb{C}^k$ has vanishing differential wherever
it is defined. To see this choose $\alpha\in\mbb{R}$ such that $\mbb{C}^{\xi_k}
\cdot x_0$ is relatively compact in $D_\alpha\Subset D$. If $\abs{t_j}$ is
sufficiently small for all $j=1,\dotsc,k-1$, the map
\begin{equation*}
\rho\circ e^{t_{1}\xi_{1}}\circ\dotsb\circ e^{t_{k-1}\xi_{k-1}}
\end{equation*}
is defined and continuous plurisubharmonic on $D_\alpha$. Consequently, this
map is constant on $\mbb{C}^{\xi_k}\cdot x_0$, and the claim follows by
induction over $k$.

Having established this, we use the following argument based on the
Campbell-Baker-Hausdorff formula (cf.~\cite[pp.31--32]{KLY}). Let
$\xi_1,\xi_2\in\wh{\lie{h}}$ and $\lambda_1,\lambda_2\in\mbb{C}$. Then we
conclude from $\rho\bigl(e^{t \lambda_{1}\xi_{1}}e^{t \lambda_{2}\xi_{2}}\cdot
x_0)=\rho(x_0)$ and $e^{t\lambda_{1}\xi_{1}}e^{t\lambda_{2}\xi_{2}}=
e^{t\lambda_{1}\xi_{1}+{t\lambda_{2}\xi_{2}}+O(t^2)}$ that $\lambda_1\xi_1
+\lambda_2 \xi_2\in\wh{\lie{h}}$ holds. Hence $\wh{\lie{h}}$ is a complex
subspace of $\lie{g}$. To see that $[\xi_1,\xi_2]$ lies in $\wh{\lie{h}}$ we use
\begin{equation*}
\rho\bigl(e^{\sqrt{t}\xi_1}e^{\sqrt{t}\xi_2}e^{-\sqrt{t}\xi_1}
e^{-\sqrt{t}\xi_2}\cdot x_0\bigr)=\rho(x_0)
\end{equation*}
together with
\begin{equation*}
e^{\sqrt{t}\xi_1}e^{\sqrt{t}\xi_2}e^{-\sqrt{t}\xi_1} e^{-\sqrt{t}\xi_2}
=e^{t[\xi_1,\xi_2]+O(t^{3/2})}
\end{equation*}
hence completing the proof that $\wh{\lie{h}}$ is a complex Lie subalgebra of
$\lie{g}$.
\end{proof}

The following example shows that even if $\wh{\lie{h}}$ is a subalgebra, without
pseudoconvexity of $D$ we do not obtain a foliation with relatively compact
leaves.

\begin{ex}
Consider the homogeneous space $X=\mbb{P}_2\setminus\bigl\{[e_1]\bigr\}\simeq
P/H$ for $P=\left\{\left(\begin{smallmatrix}*&*&*\\0&*&*\\0&*&*\end{smallmatrix}
\right)\right\}\simeq{\rm{GL}}(2,\mbb{C})\ltimes\mbb{C}^2$ and
$H=\left\{\left(\begin{smallmatrix}*&0&*\\0&*&*\\0&0&*\end{smallmatrix}
\right)\right\}\simeq(\mbb{C}^*)^2\ltimes\mbb{C}^2$. It follows from
Theorem~\ref{Thm:Hirsch} that every pseudoconvex domain spread over $X$ is
Stein. However, $X$ itself is not pseudoconvex and there does not exist a
maximal complex subgroup $\wh{H}$ of $P$ such that $\wh{H}\cdot eH$ is
relatively compact in $X$. Therefore the existence of such an $\wh{H}$ is not a
purely Lie theoretic property.
\end{ex}

After these preparations we are in the position to prove the theorem.  

\begin{proof}[Proof of Theorem~\ref{Thm:Main}]
We define $\wh{H}$ to be the analytic subgroup of $G$ with Lie algebra
$\wh{\lie{h}}$. Since $\lie{h}$ is contained in $\wh{\lie{h}}$, we have
$H^0\subset\wh{H}$. In order to show that $\dim H<\dim\wh{H}$ we suppose from
now on that the pseudoconvex domain $D$ is not Stein. It follows from
Theorem~\ref{Thm:Hirsch} that there exists $\xi\in\lie{g}$ such that
$\wt{\xi}_X(x_0)\not=0$ and such that $\mbb{C}^\xi\cdot x_0$ is relatively
compact in $D$. Since subharmonic functions on $\mbb{C}$ which are bounded from
above must be constant, we see that $\xi\in\wh{\lie{h}}\setminus\lie{h}$.
In particular, we have that $D$ is Stein if and only if $\wh{\lie{h}}=\lie{h}$.

In order to define the foliation $\mathscr{F}$ of $D$ we first construct a
relatively compact immersed complex submanifold of $D$ which contains $x_0\in D$
and which will become the leaf $F_{x_0}$. For this let $\wt{H}$ be the
simply-connected complex Lie group with Lie algebra $\wh{\lie{h}}$. As we have
seen, for every $\xi\in\wh{\lie{h}}$ the integral manifold $F^\xi_{x_0}$ is
relatively compact in $D$. This implies that we can define a map $\Phi$ on
$\exp(\wh{\lie{h}})\subset\wt{H}$ with values in $D$ such that $F^\xi_{x_0}$ is
the image of  $\exp(\mbb{C}\xi)$ under $\Phi$. Since $\exp(\wh{\lie{h}})$ is
dense in $\wt{H}$ (see~\cite{HM78}), we can extend $\Phi$ to $\wt{H}$ and
obtain an equivariant holomorphic map $\wt{H}\to D$ whose image is an immersed
complex submanifold $F_{x_0}$. Note that $\rho$ is constant on $F_{x_0}$, so
that $F_{x_0}$ is relatively compact in $D$. Moreover, it follows from the
definition of $F_{x_0}$ that every inner integral curve of $D$ passing through
$x_0$ is contained in $F_{x_0}$.

We define the foliation $\mathscr{F}$ of $D$ by moving around $F_{x_0}$ with
the local $G$-action on $D$. To make this precise, note that for every $x\in
D$ there exists elements $g_1,\dotsc,g_k\in G$ such that $g_1\cdot\bigl(\dotsb
(g_k\cdot x_0)\bigr)$ is defined and equals $x$. Then we set
\begin{equation*}
F_x:=g_1\cdot\bigl(\dotsb(g_k\cdot F_{x_0})\bigr).
\end{equation*}
The product $g_k\cdot F_{x_0}$ on the right hand side is defined since $F_{x_0}$
is relatively compact. The foliation is well-defined because of the following
observation. Suppose we have $x_0=g_1\cdot\bigl(\dotsb(g_k\cdot x_0)\bigr)$.
Then $F_{x_0}$ and $g_1\cdot\bigl(\dotsb(g_k\cdot F_{x_0})\bigr)$ must coincide
since $F_{x_0}$ contains every inner integral curve of $D$ passing through $x_0$
as noted above. One checks locally that $\mathcal{F}$ indeed defines a foliation
of $D$ fulfilling the first three statements of Theorem~\ref{Thm:Main}. For the
last statement note the following: Let $p \colon D \rightarrow X$  be a locally
biholomorphic map between complex manifolds, $A \subset D$ be a relatively compact
set and $B:= \pi(A)$. Then for every $b \in B$ the set $p^{-1}(b) \cap A$ is
finite. It follows that the restriction of $p$ to a leaf in $D$ is a finite
covering map onto its image in $X$.
\end{proof}
 
\begin{ex}
A simple method to produce examples for Theorem~\ref{Thm:Main} is the following.
Let $Y=G/L$ be a complex homogeneous space and let $H$ be a closed complex
subgroup of $L$ such that $L/H$ is compact. Then $X=G/H$ contains many
pseudoconvex non-Stein domains, e.g., pre-images of balls in $Y$ under the
fibration $G/H\to G/L$. For these domains we have $\wh{H}=L$. Concretely,
consider an embedding of $L={\rm{SL}}(2,\mbb{C})$ into $G={\rm{SL}}(3,\mbb{C})$
and let $H$ be a discrete cocompact subgroup of $L$. In this case, the
homogeneous space $G/H$ is even holomorphically convex.
\end{ex}

Since $H^0$ normalizes the group $\wh{H}$ it is clear that there exists a
maximal open subgroup of $H$ with this property. Defining
\begin{equation*}
\wt{D}:=D\times_X G:=\bigl\{(x,g)\in D\times G;\ p(x)=gH\bigr\}
\end{equation*}
we have the commutative diagram
\begin{equation*}
\xymatrix{
\wt{D} \ar[r]^{\wt{p}} \ar[d]_{\wt{\pi}} & G\ar[d]^{\pi}\\
D \ar[r]_p & X,
}
\end{equation*}
where $\wt{p}\colon\wt{D}\to G$ is locally biholomorphic and $\wt{\pi}\colon
\wt{D}\to D$ is a principal $H$-bundle over $D$. Note however that $\wt{D}$ is
not a domain spread over $G$ since it is in general not connected. Choose the
reference point $\wt{x}_0:=(x_0,e)\in\wt{D}$ and let $\wt{D}^0$ be the connected
component of $\wt{D}$ containing $\wt{x}_0$. Then we define the open subgroup
\begin{equation}\label{Eqn:Hstar}
H^*=\bigl\{h\in H;\ h\cdot \wt{x}_0\in\wt{D}^0\bigr\}
\end{equation}
of $H$. In particular, if $D=X=G/H$ is pseudoconvex, then $\wt{D}\simeq G$ and
$H^*=H$. Hence, the group $\wh{H}$ is normalized by $H$ as we shall see in

\begin{prop}\label{Prop:normalizing}
The group $H^*$ defined in equation~\eqref{Eqn:Hstar} normalizes $\wh{H}$, and
$D$ is biholomorphic to $\wt{D}^0/H^*$ which can be realized as a domain spread
over $X^*:=G/H^*$. Moreover, we have $\wh{H}=\wh{H^*}$. In other words, after
replacing $X=G/H$ by a covering we may suppose that $H$ normalizes $\wh{H}$.
\end{prop}

\begin{proof}
In a first step one checks that $h^{-1}\cdot\wt{x}_0\in\wt{D}^0$ for all $h\in
H^*$. This implies then that $H^*$ is indeed a subgroup of $H$ and that
$H^*=\bigl\{h\in H;\ h\cdot\wt{D}^0=\wt{D}^0\bigr\}$. Moreover, $H^0$ is
contained in $H^*$. Hence $H^*$ is open in $H$, closed in $G$, and acts properly
and freely on $\wt{D}^0$. By definition, we have $\wt{\pi}^{-1}\bigl(\wt{\pi}(x)
\bigr)\cap\wt{D}^0=H^*\cdot x$ for all $x\in\wt{D}^0$, and thus
$\wt{D}^0/H^*\simeq D$. Finally, one sees directly that the map $\wt{D}^0\to
D^*:=D\times_X(G/H^*)$, $(x,g)\mapsto(x,gH^*)$, induces an isomorphism between
$\wt{D}^0/H^*$ and the connected component of $D^*$ containing $(x_0,eH^*)$, so
that $D$ can be realized as a domain spread over $G/H^*$.

To finish the proof we must show that $H^*$ normalizes $\wh{H}=\wh{H^*}$. For
this we define the set
\begin{equation*}
\Omega_D:=\bigl\{(x,g)\in\wt{D}^0;\ (x,g\wh{H})\subset\wt{D}^0\text{ and }
\wt{\pi}(x,g\wh{H})\text{ is relatively compact in $D$}\bigr\}.
\end{equation*}
Since $\wt{\pi}(x_0,\wh{H})=F_{x_0}$ is relatively compact in $D$, we see that
$\Omega_D$ is a non-empty open subset of $\wt{D}^0$. Moreover, since the
leaves of $\mathscr{F}$ are relatively compact in $D$, the set $\Omega_D$ is
also closed in $\wt{D}^0$. Hence $\Omega_D=\wt{D}^0$. This implies that the
groups $H^*$ and $\wh{H}$ act on $\wt{D}^0$ by
\begin{equation*}
h\cdot(x,g):=(x,gh^{-1}).
\end{equation*}
Consequently, for each $h\in H^*$ we have $(x_0,h\wh{H}h^{-1})\subset\wt{D}^0$
and can conclude that
\begin{equation*}
\wt{\pi}(x_0,h\wh{H})=\wt{\pi}(x_0,h\wh{H}h^{-1})
\end{equation*}
is relatively compact in $D$. Since $x_0\in\wt{\pi}(x_0,h\wh{H})$ holds, we
conclude from the maximality of $F_{x_0}$ that $F_{x_0}=\wt{\pi}(x_0,h\wh{H})$.    
Hence $h\wh{H}h^{-1}=\wh{H}$.
\end{proof}

The following example shows that in general $\wh{H}$ is not normalized by the
whole group $H$.

\begin{ex}
Let $G:= {\rm{SL}}(2,\mbb C)$ and $H={\rm{SL}}(2,\mbb Z) \subset G$. Since
$H$ is Zariski dense in $G$ there is no proper connected complex subgroup of $G$
normalized by $H$. Let $X:= G/H$,
\begin{equation*}
A:=\begin{pmatrix}2&1\\1&1\end{pmatrix}\in G,
\end{equation*}
and let $\wh H \simeq \mbb C^*$ denote the Zariski closure of $H^* := \{A^k \mid k
\in \mbb Z \} \subset H$ in G. Then the orbit of $\wh H$ through the point
$eH$ in $X$ is an elliptic curve, say, $E$. Let $B \subset G$ be a Borel
subgroup transversal to $\wh H$. Then there is a small open neighborhood $U
\subset B$ biholomorphic to the unit ball $\mbb B_2$ in $\mbb C^2$, such that
$D:=U \cdot E \subset X$ is isomorphic to $\mbb B_2 \times E$. Therefore $D$ is
a pseudoconvex domain in $X$ and we see that $H\wh H$ is not a subgroup of $G$.
Taking the covering $\pi\colon Y:= G/H^*\to G/H = X$, one sees that the
map
\begin{equation*}
\pi\vert_{{\pi}^{-1}(D)}\colon{\pi}^{-1}(D) \to D
\end{equation*}
is biholomorphic. So by taking the open subgroup $H^*$ of $H$ and the associated
covering, we do not change the pseudoconvex domain $D$ but we now have that
$H^*\wh H$ is a group.   
\end{ex}

As an application of Theorem~\ref{Thm:Main} we have the following

\begin{cor} 
Let $p\colon D\to X$ be a pseudoconvex domain over $X=G/H$ such that $eH\in p(D)$.   
If the subgroup $H$ is connected and maximal in $G$, then $D$ is either
compact or Stein.
\end{cor}  

\begin{proof} 
Since $H \subset \wh{H} \subset G$ and $H$ is maximal, either $\wh{H}=G$ or
$\wh{H}=H$. In the first case $D$ itself consists of exactly one leaf of the
foliation and thus is compact. Otherwise $\wh{H}=H$, i.e., $D$ contains no
inner integral curve, and so $D$ is Stein.
\end{proof}

Note that the isotropy subgroups of projective space and also of $Gr(k,n)$, the
Grassmann manifold of $k$-dimensional subspaces of an $n$-dimensional vector
space, are connected and maximal. Hence we have reproduced some classical
results, e.g., see \cite{Fuj63}, \cite{Hir75}, \cite{Nis62}, \cite{Tak64} and
\cite{Ue80}.

\begin{cor}
Let $G$ be a simple complex Lie group and $\Gamma$ a discrete Zariski dense
subgroup of $G$. If $X=G/\Gamma$ is pseudoconvex, then $X$ is compact.
\end{cor}

\begin{proof}
Suppose that $X=G/\Gamma$ is pseudoconvex. Since $X$ cannot be Stein, there is a
connected Lie subgroup $\wh{\Gamma}\subset G$ of positive dimension such that
$\wh{\Gamma}\cdot x_0$ is relatively compact in $X$. Since $\Gamma$ normalizes
$\wh{\Gamma}$ and is Zariski dense in $G$, we have $\wh{\Gamma}\lhd G$.  
Hence $\wh{\Gamma}=G$ which proves the claim.
\end{proof}

Let $p\colon D\to X$ be a pseudocovex domain spread over $X=G/H$ with $x_0\in
p(D)$. For later use we note the following technical

\begin{lem}\label{Lem:Intersection}
Let $\wt{G}$ be a connected closed complex subgroup of $G$ such that
$\wt{G}\cdot x_0$ is closed in $X$. Then every connected component of
$\wt{D}:=p^{-1}(\wt{G}\cdot x_0)$ is a pseudoconvex domain spread over
$\wt{G}\cdot x_0\simeq\wt{G}/(\wt{G}\cap H)$ and we have
\begin{equation*}
\wh{\wt{G}\cap H}=(\wt{G}\cap\wh{H})^0
\end{equation*}
where the left hand side is the connected subgroup of $\wt{G}$ corresponding to
$\wt{D}\to \wt{G}/(\wt{G}\cap H)$.
\end{lem}

\begin{proof}
Since $\wt{D}$ is a closed complex submanifold of $D$, all of its connected
components are pseudoconvex domains spread over $\wt{G}\cdot x_0$ by the map
$\wt{p}:=p|_{\wt{D}}$. By definition, the Lie algebra of $\wh{\wt{G}\cap H}$ is
\begin{equation*}
\wh{\wt{\lie{g}}\cap\lie{h}}:=\bigl\{\xi\in\wt{\lie{g}};\
\wt{\xi}_X\varphi(x_0)=0\text{ for every continuous plurisubharmonic function
$\varphi$ on $\wt{D}$}\bigr\},
\end{equation*}
hence contains $\wt{\lie{g}}\cap\wh{\lie{h}}$. Conversely, every element of
$\wh{\wt{\lie{g}}\cap\lie{h}}$ induces an inner integral curve in $\wt{D}$, thus
also in $D$ since $\wt{D}$ is closed in $D$. This implies $\wh{\wt{\lie{g}}\cap
\lie{h}}=\wt{\lie{g}}\cap\wh{\lie{h}}$ as was to be shown.
\end{proof}

\section{A characterization of holomorphic convexity}\label{Section:holconv}

If a pseudoconvex non-Stein domain $D$ is spread over $X$, in general, the leaf
$F_{x_0}$ is not closed in $D$. In this section we investigate exactly when this
happens.

\begin{thm} \label{cldwhH}
Suppose $D$ is a pseudoconvex domain spread over the complex homogeneous
manifold $G/H$. Then the complex group  $H^*\wh{H}$ constructed in
Proposition~\ref{Prop:normalizing} is closed in $G$ if and only if $D$ is
holomorphically convex.

Moreover, when these conditions hold, the Remmert reduction of $D$ is a
holomorphic fiber bundle $\wh{\pi}\colon D\to D_0$ that is induced by the bundle
$\pi\colon G/H\to G/H^*\wh{H}$. The fiber of $\wh{\pi}$ is compact and is
biholomorphic to $H^*\wh{H}/H_1$, where $H_1$ is a subgroup of $H^*$ having
finite index. Its base $D_0$ is a Stein domain spread over the homogeneous
manifold $G/H^*\wh{H}$.
\end{thm}

\begin{proof}
Since $D$ is biholomorphic to a connected component of $D\times_XG/H^*$ we may
assume $H=H^*$ throughout the proof in order to simplify the notation. Notice
that $H\wh{H}$ is closed in $G$ if and only if the leaves of $\mathscr{F}$ are
compact.

We first consider what happens when $D \subset X$ is a domain in the homogeneous
manifold with the group $H\wh{H}$ closed. In this case the fibration $\pi\colon
G/H\to G/H\wh{H}$ induces the foliation $\mathscr{F}$ of $D$ that we constructed
in Theorem \ref{Thm:Main}. Since the plurisubharmonic exhaustion function on $D$
is constant on the compact $\wh{H}$-orbits in $D$, it descends to $\pi(D)\subset
G/H\wh{H}$. Hence, $D_0:=\pi(D)$ is pseudoconvex and by the maximality of
$\wh{H}$ a Stein domain. Moreover, we have $H_1=H$ in this case.

In order to be able to repeat this argument in the general case, we need to find
a domain $D_0$ spread over $G/H\wh{H}$ such that the diagram
\begin{equation*}
\xymatrix{
D\ar[r]^p\ar[d] & G/H\ar[d]\\
D_0\ar[r]_{p_0} & G/H\wh{H}
}
\end{equation*}
commutes. The idea is to define $D_0:=D/\mathscr{F}$. Since the leaves of the
foliation $\mathscr{F}$ are compact, by~\cite[Proposition~6.2]{Hol78} the leaf
space $D/\mathscr{F}$ carries a canonical complex structure as soon as it is
Hausdorff.

In order to see that $D/\mathscr{F}$ is indeed Hausdorff let $F_x$ be the leaf
through $x\in D$ and let $U$ be any open neighborhood of it. We must show that
$U$ contains a saturated open neighborhood of $F_x$. Since the domain $p(D)$ is
foliated by the images $p(F_x)$, $F_x\in\mathscr{F}$, and since this foliation
is induced by the fiber bundle $G/H\to G/H\wh{H}$ having compact fibers we find
a saturated open neighborhood $V$ of $p(F_x)$ inside $p(U)$, such that the
connected component $W$ of $p^{-1}(V)$ containing $F_x$ lies in $U$ and covers
$V$. Now $W$ is a saturated open neighborhood of $F_x$.

Having established the existence of the commutative diagram it follows that
$D_0$ is Stein in exactly the same way as above. Consequently, the quotient map
$D\to D/\mathscr{F}=D_0$ is the Remmert reduction of $D$ whose fibers are the
leaves. Since $F_{x_0}$ is connected, we have $F_{x_0}\simeq(H\wh{H})/H_1$ where
$H_1$ is the smallest open subgroup of $H$ such that $H\wh{H}=H_1\wh{H}$ holds.

Conversely, if $D$ is holomorphically convex, then it has a Remmert reduction, 
i.e., there exists a holomorphic map $\sigma\colon D \to D_1$ that has compact 
connected fibers and the target space $D_1$ is Stein. Since $D$ is pseudoconvex,
it admits a plurisubharmonic exhaustion. This function is clearly constant on
the fibers of the map $\sigma$. It is clear from the construction given in
Theorem~\ref{Thm:Main} that the foliation given by the subgroup $\wh{H}$ gives
the same partition of $D$ as is given by the fibers of the map $\sigma$. The
fact that $D$ is the total space of a holomorphic fiber bundle follows from our
considerations above and this observation.
\end{proof}

\section{Domains spread over projective orbits}

In this section we again consider $p\colon D\to X$ which is a pseudoconvex
domain spread over a homogeneous manifold $X=G/H$, as in Theorem~\ref{Thm:Main},
but now assume that $X$ is an orbit of the connected complex linear group $G$ in
some projective space $\mathbb P_N$. One is in an algebraic setting, if $X$ is
compact and thus a flag manifold. In very stark contrast to this, if $X$ is not
compact, then $X$ need not be closed or even locally closed in $\mathbb P_N$ and
the setting is not algebraic. Nonetheless, there are specific facts at hand
concerning the {\it holomorphic actions} of complex groups in the projective
case that allow us to prove the next result.

\begin{thm}\label{projcase}   
Let $X=G/H$ be an orbit of a connected complex Lie group $G$ acting
holomorphically on some projective space $\mathbb P_N$. Then any pseudoconvex
domain $D$ spread over $X$ is holomorphically convex. Moreover, the fibers of
the Remmert reduction of $D$ given in Theorem~\ref{cldwhH} are homogeneous
rational manifolds that are biholomorphic to $H\wh{H}/H$.    
\end{thm}  

\begin{proof}  
If $D$ itself is Stein, then $\wh{H}=H$ and there is nothing to prove in this
case. So we assume throughout the rest of the proof that $D$ is not Stein.

The first step of the proof consists in reducing the general situation to an
algebraic one. Denoting $\ol{G}$ the algebraic Zariski closure of the image of
$G$ in ${\rm{PGL}}(N+1,\mbb{C})$ and $G'$ its commutator subgroup, we have
$G'=\ol{G}'$, and in particular $G'$ is algebraic (for a proof of this result of
Chevalley see~\cite[Corollary~II.7.9]{Bo}). Consequently, the boundary of every
$G'$-orbit in $X$ consists of $G'$-orbits of strictly smaller dimension. Since
$G'$ is a normal subgroup of $G$, this implies that every $G'$-orbit is closed
in $X$, and in particular, $G'H$ is a closed subgroup of $G$.

We claim that the relatively compact orbit $\wh{H}\cdot y_0$ is contained in the
neutral fiber of $G/H\to G/(G'H)$. To see this, we will repeat an argument
from~\cite[p.~173]{HuOe}. Since $G\cap(\ol{G}_{y_0}\ol{G}')=G\cap(\ol{G}_{y_0}
G')=HG'$, the Abelian algebraic group  $\ol{G}/(\ol{G}_{y_0}\ol{G}')\simeq
\mbb{C}^k\times(\mbb{C}^*)^l$ contains $G/(HG')$ as a $G$-orbit. As a
consequence, the fiber bundle $G/H\to G/(HG')$ has holomorphically separable
base which proves the claim. Moreover, since $G'$ acts transitively on this
fiber, we have $(HG')/H\simeq G'/(H\cap G')$ and $\dim\wh{H}\cdot
y_0=\dim\wh{H}_1\cdot y_0$ where $\wh{H}_1:=\wh{H}\cap G'$ due to
Lemma~\ref{Lem:Intersection}. Note that as a fiber $G'/(H\cap G')$ is closed in
$G/H$. Thus $\wh{H}_1\cdot y_0$ is still relatively compact in the
quasi-projective variety $G'/(H\cap G')$. Therefore we may replace $G$, $H$ and
$\wh{H}$ by $G_1:=G'$, $H_1:=G'\cap H=G'_{y_0}$ and $\wh{H}_1$ respectively.
Now we can iterate this procedure, thus replacing $G_1$ by $G_2:=G_1'$, and so on. 
As above we keep the orbit $\wh{H}\cdot y_0$ as a relatively compact subset in
$G_2/H_2$. Therefore this iteration must terminate after finitely many steps and
we end up with an algebraic group $G_k$ such that $G_k=G_k'$. Consequently, we
may assume without loss of generality that $X=G/H$ is a quasi-projective variety
containing a relatively compact orbit $\wh{H}\cdot y_0$ and that $G=G'$.

Since $H$ is an algebraic subgroup of $G$, the map $G/H^0\to G/H$ is a finite
covering. Thus there exists a finite proper map between each connected component
of $D\times_X(G/H^0)$ to $D$, which implies that $D$ is pseudoconvex or
holomorphically convex if and only if each component of $D\times_X(G/H^0)$ has
this property. Therefore  we may assume that $H$ is connected. In particular
this implies $H\subset\wh{H}$. Hence it suffices to show that $\wh{H}$ is closed
in $G$. If this is not the case, let $L_1$ be the topological closure of
$\wh{H}$ and let $G_1$ be the connected Lie subgroup of $G$ having Lie algebra
$\lie{l}_1+i\lie{l}_1$. It follows that $\wh{\lie{h}}=\lie{l}_1\cap i\lie{l}_1$,
therefore $\wh{\lie{h}}\lhd\lie{g}_1$. We iterate this procedure until we arrive
at a group $G_n$ which is closed in $G$. Then we have $H\subset\wh{H}\subset
G_n$ and $G_n/H$ closed in $G/H$, so that we can apply the sequence of
commutator fibrations to $G_n/H$. Again these two reduction procedures must
terminate after finitely many steps. Hence replacing $G$ by the group $G_n$
finally obtained we are in the situation that $H\subset\wh{H}\subset G=G'$, that
$X=G/H$ is quasi-projective and that there is a sequence of connected complex
Lie subgroups
\begin{equation*}
\wh{H}=:G_0\lhd G_1\lhd  \dotsb \lhd G_n=G
\end{equation*}
such that $\lie{g}_j=\lie{l}_j+i\lie{l}_j$ and $\lie{g}_{j-1}=\lie{l}_j\cap
i\lie{l}_j$ where $L_j$ is the topological closure of $G_{j-1}$ for all
$j=1,\dotsc,n$.

Now a purely algebraic argument gives
\begin{equation*}
\lie{l}_n\supset [\lie{g},\lie{l}_n]=\lie{l}_n'+i\lie{l}_n'=\lie{g}'=\lie{g}.
\end{equation*}
Consequently we obtain $\lie{g}=\lie{l}_n$, hence $\lie{g}=\lie{g}_{n-1}$.
Repeating this, we see $\wh{\lie{h}}=\lie{g}$, so that $\wh{H}$ is an algebraic
subgroup of $G$ and in particular closed as was to be shown.
\end{proof}

\begin{rem}
Let $H\subset G$ be linear algebraic groups and $X=G/H$. The proof of
Theorem~\ref{projcase} shows that if there is a non-Stein pseudoconvex domain
$p\colon D\to X$ spread over $X$ with $eH\in p(D)$, then the group $\wh{H}$
constructed in Theorem~\ref{Thm:Main} is likewise an algebraic subgroup of $G$.
\end{rem}

In passing, we also note what happens in the case of complex orbits of real groups
acting holomorphically on projective space.

\begin{cor} 
Let $G_\mbb{R}$ be a real subgroup of ${\rm{PSL}}(N+1,\mathbb C)$ that is acting
holomorphically and effectively on $\mathbb P_N$. Let $X := G_\mbb{R}\cdot x$ be
a complex orbit of some point $x\in\mathbb P_N$. Then any pseudoconvex domain
spread over $X$ is holomorphically convex.
\end{cor}

\begin{proof}
Let $G$ denote the smallest connected complex subgroup of ${\rm{PSL}}
(N+1,\mathbb C)$ that contains $G_\mbb{R}$. Then $X$ is open in $G\cdot x$.
Hence any domain spread over $G_\mbb{R}\cdot x$ is also a domain spread over
$G\cdot x$. The result now follows from Theorem~\ref{projcase}.
\end{proof}

For a general complex homogeneous manifold $X=G/H$ we have the normalizer
fibration $X=G/H\to G/\mathscr{N}_G(H^0)$ whose base is an orbit of the linear
Lie group $\Ad(G)$ in a projective space. However, even if $G/H$ is
pseudoconvex, in general $G/\mathscr{N}_G(H^0)$ does not have to be.  
    
The following examples show that one cannot control how $\wh{H}$ is related to
$\mathscr{N}_G(H^0)$.

\begin{ex}
Let $G={\rm{SL}}(3,\mbb{C})$ with Borel $B$ and a maximal parabolic subgroup
$P$. Then $G/B\to G/P$ is a $\mbb{P}_1$-bundle over $\mbb{P}_2$. Taking the
inverse image of e.g. the unit ball in $\mbb{P}_2$, we obtain a pseudoconvex
non-Stein domain in $G/B$ such that $\wh{B}=P$ which is not contained in
$\mathscr{N}_G(B)=B$.
\end{ex}

\begin{ex}
Let $S:={\rm{SL}}(3,\mathbb C)$ and $B$ be a standard Borel in $S$. We denote by
$T$ a maximal algebraic torus in $S$ and take a holomorphic proper injection of
$\mbb{C}$ into $T$ as a closed subgroup $A_1$ such that the quotient $T/A_1=:E$
is an elliptic curve. Set $H_1=A_1\ltimes U$, where $U$ denotes the unipotent
radical of $B$. Then one has the homogeneous fibration $S/H_1\to S/B$ with
compact fiber $E$. Suppose $D_1\subset S/H_1$ is a pseudoconvex domain that is not
Stein. Then $\wh{H}_1=B=\mathscr{N}_S(H_1)$ and ${\mathscr O}(S/H_1)
\simeq\mbb{C}$.
\end{ex}

\begin{ex}
With the same set up as in the previous example we now take $A_2$ to be the
closed image of a representation of $\mathbb Z$ into $T$ such that $T/A_2$ is a
Cousin group $C$ and set $H_2 := A_2 \ltimes U$. Then $S/H_2 \to S/B$ is a
homogeneous fibration with the Cousin group as fiber. Let $D_2 \subset S/H_2$ be
a pseudoconvex domain that is not Stein. Then $\wh{H}_2=\mbb{C}\ltimes U$, where
the image of $\mbb{C}$ is one of the complex leaves of the foliation of $C$,  
while $\mathscr{N}_S(H_2^0)=B$.
\end{ex}

\section{Pushing down plurisubharmonic functions}

In this section we prove the following technical result.

\begin{lem}\label{Lem:Kiselman}
Let $p\colon X\to Y$ be a holomorphic fiber bundle of complex manifolds,   
where $X$ is pseudoconvex.   
Then $Y$ is also pseudoconvex if either of the following holds:  
\begin{enumerate}  
\item the fiber is  a Cousin group $C$  
\item the bundle is a principal $(\mbb{C}^{*})^{k}$--bundle.  
\end{enumerate}   
\end{lem}

\begin{proof}   
We consider case (1) first.   
Let $\rho\colon X\to\mbb{R}$ be a continuous plurisubharmonic exhaustion of
$X$. We define a function $\rho_Y\colon Y\to\mbb{R}$ by
\begin{equation*}
\rho_Y(y):=\inf\bigl\{\rho(x);\ x\in p^{-1}(y)\bigr\}.
\end{equation*}
As an exhaustion $\rho$ attains a minimum on every closed complex submanifold of
$X$. In particular, we see that $\rho_Y$ is indeed real-valued.

Let us first show that $\rho_Y$ is continuous. For this let $(y_n)$ be a
sequence which converges to $y_0\in Y$. By the above remark there exist
elements $x_n,x_0\in X$ such that $\rho_Y(y_n)=\rho(x_n)$ and $\rho_Y(y_0)=
\rho(x_0)$. For every $\eps>0$ the set $U_{\eps}:=\bigl\{\rho<\rho(x_0)+\eps
\bigr\}$ is open and relatively compact in $X$ and contains $x_0$ and hence
almost all $x_n$. Therefore, there is $z\in\ol{U}_{\eps}$ such that $x_n\to z$
for a subsequence. We have $p(z)=\lim p(x_n)=y_0$ and consequently
$\rho(z)\geq\rho_Y(y_0)$. On the other hand,
\begin{equation*}
\rho_Y(y_n)=\rho(x_n)\to\rho(z)\leq\rho(x_0)+\eps=\rho_Y(y_0)+\eps
\end{equation*}
for all $\eps>0$ by continuity of $\rho$ which implies $\lim\rho_Y(y_n)=
\rho_Y(y_0)$ as was to be shown.

By continuity, $\{\rho_Y\leq c\}$ is closed in $Y$ for every $c\in\mbb{R}$. One
checks directly that $\{\rho_Y\leq c\}$ is contained in $p\{\rho\leq c\}$ which
implies that $\{\rho_Y\leq c\}$ is compact, hence that $\rho_Y$ is exhaustive.

Finally we show that $\rho_Y$ is plurisubharmonic. Since this can be checked
locally, let $U\subset Y$ be open and isomorphic to the unit ball $\mbb{B}_n$ so
that $p^{-1}(U)\simeq \mbb{B}_n\times C$ where $C\simeq\mbb{C}^k/\Gamma_{k+l}$ is
a Cousin group. For fixed $z\in\mbb{B}_n$ let $\rho_z$ be the plurisubharmonic
function $C\to\mbb{R}$, $g\mapsto\rho(z,g)$. Its pull-back to $\mbb{C}^k$ is
$\Gamma_{k+l}$-invariant and plurisubharmonic. Since the image of the complex
vector subspace $V:=\langle\Gamma_{k+l}\rangle_\mbb{R}\cap
i\langle\Gamma_{k+l}\rangle_\mbb{R}$ of $\mbb{C}^k$ in $C$ is an immersed complex
submanifold which is dense in the compact torus
$\langle\Gamma_{k+l}\rangle_\mbb{R}/\Gamma_{k+l}$, the pull-back of $\rho_z$ is
invariant under $\langle\Gamma_{k+l}\rangle_\mbb{R}$. Hence, it pushes down to a
plurisubharmonic function $\ol{\rho}_z$ on $\mbb{C}^k/V$ which is still invariant
under $\langle\Gamma_{k+l}\rangle_\mbb{R}/V$. Since
$\langle\Gamma_{k+l}\rangle_\mbb{R}/V$ is a real form of $\mbb{C}^k/V$, we may
apply Kiselman's minimum principle (see~\cite{Kis}) to the plurisubharmonic
function $\ol{\rho}\colon\mbb{B}_n\times(\mbb{C}^k/V)\to\mbb{R}$ and obtain
plurisubharmonicity of
\begin{equation*}
z\mapsto \rho_Y(z)=\inf\bigl\{ \ol{\rho}(z,w);\ w\in\mbb{C}^k/V\bigr\}
\end{equation*}
as was to be shown.  

In the second case we may apply 
Kiselman's minimum principle to an \emph{$(S^{1})^{k}$-invariant}
plurisubharmonic exhaustion of $X$ and 
essentially repeat the above argument.   
\end{proof}

Example~\ref{Ex:CousinPrincBundle} will show that pseudoconvexity
of $Y$ does not imply pseudoconvexity of $X$.

\section{Pseudoconvex reductive homogeneous spaces}

Recall that the holomorphic reduction of $X=G/H$ is given by $\pi\colon G/H\to
G/J$ where $J$ is a closed complex subgroup of $G$ containing $H$ such that
$G/J$ is holomorphically separable and $\mathscr{O}(G/H)\simeq
\pi^*\mathscr{O}(G/J)$. More precisely, one has
\begin{equation*}
J=\bigl\{g\in G;\ f(gH)=f(eH)\text{ for all $f\in\mathscr{O}(G)$}\bigr\}.
\end{equation*}
If $X=G/H$ is holomorphically convex, then the holomorphic reduction $X=G/H\to
Y=G/J$ coincides with the Remmert reduction of $X$, i.e., $Y$ is Stein and the
fiber $J/H$ is connected and compact. Conversely, if $X=G/H$ admits an
equivariant map onto a Stein manifold with connected compact fibers, then $X$
is holomorphically convex.

If $D$ is a pseudoconvex non-Stein domain in $X=G/H$ containing $x_0=eH$, then
the complex subgroup $\wh{H}$ constructed in Theorem~\ref{Thm:Main} must
be contained in $J$. However, even if $D=X=G/H$ is pseudoconvex non-Stein, the
group $\wh{H}$ does not necessarily coincide with $J$ as the example of a 
non-compact Cousin group shows.

In this section we suppose that $G$ is connected complex reductive. We shall see
that the base of the holomorphic reduction of a pseudoconvex $X=G/H$ is Stein. We
begin with the case that $G$ is semisimple.

\begin{thm} \label{Thm:semisimple}  
Let $X=G/H$ be a complex homogeneous manifold with $G$ a complex semisimple
Lie group. If $X$ is pseudoconvex, then $X$ is holomorphically convex.   
\end{thm}

\begin{proof}
Due to~\cite{BO} the base of the holomorphic reduction $\pi\colon X=G/H\to G/J$
is quasi-affine and $J$ is an algebraic subgroup of $G$. Hence, $\pi$ factorizes
as
\begin{equation*}
\xymatrix{
X=G/H\ar[rr]^{\pi}\ar[dr] & & G/J\\
& G/\ol{H}\ar[ur]
}
\end{equation*}
where $\ol{H}$ is the Zariski-closure of $H$ in $G$. Moreover, we have
$\mathscr{O}(G)^H=\mathscr{O}(G)^{\ol{H}}$, see~\cite{BO}. This shows that
$G/J$ is also the holomorphic reduction of $G/\ol{H}$.

Let $\varphi\colon X \to \mathbb R$ be a plurisubharmonic exhaustion
function. By \cite{Bert}, \cite{BerO} we get that $\varphi$, considered as a
function on $G$, is already invariant under the right $\ol H$-action on $G$.
Therefore $\varphi$ pushes down to a plurisubharmonic exhaustion function on the
homogeneous quotient of algebraic groups $G/{\ol H}$ and $\ol{H}/H$ is compact.
Then Theorem~\ref{projcase} gives the existence of an algebraic reductive
subgroup $\wh{H}\subset G$ containing $\ol H$ with $\wh{H}/\ol{H}$ compact. Note
that $\wh{H}=J$. Considering now the fibration $G/H \to G/\wh H$, the claim
follows. Here we used implicitly the fact that quotients of reductive groups are
Stein, if and only if the isotropy is also reductive (see \cite{Mat}, \cite{On}).
\end{proof}

\begin{ex}\label{Ex:CousinPrincBundle}
Let us give an example of a non-pseudoconvex semisimple manifold. Let
$G={\rm{SL}}(3,\mbb{C})$ and take $\Gamma\simeq\mbb{Z}$ as a discrete subgroup
of its maximal torus $T\simeq\mbb{C}^*\times\mbb{C}^*$ such that $T/\Gamma$ is a
non-compact Cousin group. Then we have $\ol{\Gamma}=T$ and the holomorphic
reduction of $X=G/\Gamma$ is the Stein manifold $G/T$. However, $X=G/\Gamma$ is
not pseudoconvex since the fiber $T/\Gamma$ is not compact.
\end{ex}

The above theorem does not hold in the reductive case as the simple example of a
Cousin groups shows. The following non-abelian and (in general)
non-holomorphically convex examples of pseudoconvex reductive manifolds indicate
their complexity.

\begin{ex}\label{Ex:reductive}
Let $G:={\rm{SL}}(2,\mathbb C)\times\mathbb C^*$ and $H\simeq\mathbb C^*\times
\mathbb C^*$ the subgroup of $G$ given by the product of the diagonal matrices
$D$ in ${\rm{SL}}(2,\mathbb C)$ and the second factor in $G$. Let $\Gamma \simeq
\mathbb Z$ be a discrete subgroup of $H$ and $J \subset H$ the smallest connected
real subgroup containing $\Gamma$ and the maximal compact subgroup of $H$. Then
$J\simeq S^1 \times S^1 \times \mathbb R$. Consider the complex homogeneous space
$X:=G/\Gamma$. We shall see that $X$ is pseudoconvex.

If $\dim_{\mathbb R}{\rm{SL}}(2,\mathbb C)\cap J =1$, we get that the real
fibration $X=G/\Gamma\rightarrow G/J$ has compact fibers and that
${\rm{SL}}(2,\mathbb C)$ acts transitively on the base $G/J={\rm{SL}}(2,\mathbb
C)/K$ where $K \subset D$ is the compact diagonal in ${\rm{SL}}(2,\mathbb C)$.
Therefore a left $K$-invariant plurisubharmonic exhaustion function on
${\rm{SL}}(2,\mathbb C)$ induces a plurisubharmonic exhaustion function on $X$.
Note that there are two possibilities for the Zariski-closure of $\Gamma$: First
it might be isomorphic to $\mbb{C}^*$ in which case $X$ is an elliptic curve
bundle over ${\rm{SL}}(2,\mbb{C})$, then holomorphically convex and K\"ahler,
see~\cite[Theorem~5.1]{GMO}. On the other hand, for generic $\Gamma$, the
manifold $X$ is a holomorphic fiber bundle with non-compact Cousin fibers over
the affine quadric, so it is not holomorphically convex. Depending on whether the
projection of $\Gamma$ to the central $\mbb{C}^*$-factor is closed or not, $X$
may or may not be K\"ahler. 

If $\dim_{\mathbb R} {\rm{SL}}(2,\mathbb C) \cap J =2$, then  $\Gamma \subset D
\subset {\rm{SL}}(2,\mathbb C)$ and $X$ is an elliptic curve bundle over the
product of the affine quadric and $\mathbb C^*$. Hence $X$ is pseudoconvex and
holomorphically convex but, since the ${\rm{SL}}(2,\mbb{C})$-orbits are not
closed in $X$, these examples are not K\"ahler.
\end{ex}

\begin{thm}\label{Thm:reductive}
Let $G$ be connected complex reductive and $X=G/H$ be pseudoconvex. Then the
holomorphic reduction $G/J$ of $X$ is Stein and we have $\mathscr{O}(J/H)
=\mbb{C}$.
\end{thm}

\begin{proof}
We define $N$ to be the union of all connected components of the normalizer
$\mathscr{N}_G(H)$ of $H$ in $G$ which meet $H$ so that $N/H$ is a connected
complex Lie group. Note that $N$ contains $Z:=\mathscr{Z}(G)^0
\simeq(\mbb{C}^*)^k$.

Consider the holomorphic principal bundle $X=G/H\to G/N$ with structure group
$N/H$ and let $N/H\to N/I$ be the holomorphic reduction of its fiber. We obtain
a new principal bundle $X=G/H\to G/I$ whose fiber $I/H$ is now a Cousin group.
Due to Lemma~\ref{Lem:Kiselman} the base $G/I$ is again pseudoconvex. Moreover,
$G/H$ and $G/I$ have the same holomorphic reduction. Suppose that $\dim G/I<\dim
G/H$ holds. Arguing by induction over $\dim G/H$ we may assume that the
holomorphic reduction of $G/I$ is Stein, hence the same is true for $X=G/H$.

Therefore we must deal with the case $\dim G/I=\dim G/H$. This implies that
$I=H$, i.e., that $N/H$ is Stein. As noted above, the group
$Z\simeq(\mbb{C}^*)^k$ acts holomorphically on $N/H$. Since the latter is Stein,
$Z$ has a closed orbit in $N/H$ and since $Z$ is normal, all of its orbits are
closed. In particular, we have $Z\cap H\simeq(\mbb{C}^*)^{l_1}$. Thus we obtain a
fibration $X=G/H\to G/(HZ)$ which is a principal bundle with structure group
$(\mbb{C}^*)^{k-l_1}$. Due to Lemma~\ref{Lem:Kiselman} its ase $G/(HZ)$ is
pseudoconvex. Since the derived group $G'$ acts transitively on $G/(HZ)$,
Theorem~\ref{Thm:semisimple} yields that $G/(HZ)$ is holomorphically convex. Let
us consider the commutative diagram
\begin{equation*}
\xymatrix{
G/H\ar[rr]\ar[dd]\ar[dr]_{\pi} & & G/\ol{H}\ar[dd]\\
 & G/\bigl(\ol{H}\cap(HZ)\bigr)\ar[ur]\ar[dl]\ar[dr] & \\
G/HZ\ar[rr] & & G/\ol{HZ}
}
\end{equation*}
where the bars denote the Zariski closure inside $G$. It is sufficient to show
that $G/\ol{H}$ is pseudoconvex since $G/\ol{H}$ is then holomorphically convex
by Theorem~\ref{projcase} and since the holomorphic reductions of $G/H$ and
$G/\ol{H}$ coincide by~\cite{BO}.

The proof of Theorem~\ref{Thm:semisimple} shows that $\ol{HZ}/HZ$ is compact.
Since $Z$ is an algebraic subgroup of $G$, we have $\ol{HZ}=\ol{H}Z$. This
implies that $\ol{HZ}/\ol{H}\simeq Z/(Z\cap\ol{H})\simeq(\mbb{C}^*)^{k-l_2}$.
Consider the fiber bundle $\ol{H}/H\to\ol{H}/\bigl(\ol{H}\cap(HZ)\bigr)$.   
The torus $\ol{H}\cap Z$ acts transitively on its fiber. Hence this fiber
satisfies $\bigl(\ol{H}\cap(HZ) \bigr)/H\simeq(\ol{H}\cap Z)/(H\cap Z)\simeq
(\mbb{C}^*)^{l_2-l_1}$. Consequently, the map $\pi$ is a
$(\mbb{C}^*)^{l_2-l_1}$-principal bundle and Lemma~\ref{Lem:Kiselman} shows
that $G/\bigl(\ol{H}\cap(HZ)\bigr)$ is pseudoconvex. Again from
$\ol{HZ}=\ol{H}Z$ we conclude that $\ol{H}$ acts transitively on $\ol{HZ}/HZ$,
hence that $\ol{H}/\bigl(\ol{H}\cap(HZ)\bigr)\simeq\ol{HZ}/HZ$ is compact.
Therefore $G/\ol{H}$ is pseudoconvex, thus $G/J$ is Stein.

We still must prove $\mathscr{O}(J/H)=\mbb{C}$. Since $G/J$ is Stein, $J$ is
reductive by \cite{Mat} or \cite{On}. Now consider the holomorphic reduction
$J/H\to J/I$. Since $J/H$ is a closed complex submanifold of $X=G/H$, it is
pseudoconvex, hence by the above $J/I$ is Stein. Therefore, $I$ is reductive and
$G/I$ is Stein. From the factorization $G/H\to G/I\to G/J$ we obtain $I=J$. Hence
$\mathscr{O}(J/H)=\mbb{C}$, as was to be shown.
\end{proof}

A compact K\"ahler $G/H$ is a product of a compact complex torus and a
homogeneous flag manifold; see Matsushima~\cite{Mat57} and
Borel-Remmert~\cite{BR}. For $G$ reductive we have the following extension of this
result for a K\"ahler pseudoconvex $G/H$ under the assumption that
$\mathscr{O}(G/H)\simeq\mathbb C$. In particular, Theorem~\ref{Thm:Kaehler}
applies to the fiber of the holomorphic reduction of any K\"ahler pseudoconvex
reductive homogeneous manifold because of Theorem~\ref{Thm:reductive}.  

\begin{thm}\label{Thm:Kaehler}
Let $G$ be connected complex reductive and $X=G/H$ be pseudoconvex and K\"{a}hler
with $\mathscr{O}(X)\simeq\mathbb C$. Then $G/\ol{H}$ is a homogeneous rational
manifold, $\ol{H}/H$ is a Cousin group, and the bundle $G/H \to  G/\ol{H}$ is
holomorphically trivial.
\end{thm}

\begin{rem}
Note that for any K\"ahler pseudoconvex reductive homogeneous manifold, since the
base $G/J$ of its holomorphic reduction $G/H\to G/J$ is Stein, $J/\ol{H}$ is
compact if and only if $G/\ol{H}$ is holomorphically convex. Moreover, $J$ is
reductive and $\wh{\lie{h}} = \lie{j}'\oplus V$ in this setting, where $\lie{j}'$
denotes the derived Lie algebra of $\lie{j}$ and $V$ is the complex vector
subspace of $\mathbb C^k$ defined in the proof of Lemma~\ref{Lem:Kiselman} with
the space $\mathbb C^k$ considered as the Lie algebra of the Cousin group $C$.
Note also that Example~\ref{Ex:reductive} shows that in general $G/H\to G/J$ is
\emph{not} holomorphically trivial.
\end{rem}

\begin{proof}
In order to prove the first claim we will show that $G/\ol{H}$ is pseudoconvex.
Since $X$ is K\"ahler, $G'\cap H$ is algebraic by~\cite[Theorem~5.1]{GMO}. Note
that $G'\cap H$ is a closed normal subgroup of $H$ and that $H/(G'\cap H)$ is an
Abelian complex Lie group. Since $G'\cap H$ is algebraic, $\mathscr{N}_G(G'\cap
H)$ is also an algebraic subgroup of $G$, hence must contain $\ol{H}$. Thus
$\ol{H}/(G'\cap H)$ is an affine algebraic group containing $H/(G'\cap H)$ as a
Zariski-dense closed complex subgroup. This implies that $\ol{H}/(G'\cap H)$ is
Abelian as well.   Hence
\begin{equation*}
\ol{H}/H\cong\bigl(\ol{H}/(G'\cap H)\bigr)/\bigl(H/(G'\cap H)\bigr)
\end{equation*}
is an Abelian complex Lie group which cannot have a factor isomorphic to
$\mbb{C}$. Therefore we have shown that $G/H\to G/\ol{H}$ is a holomorphic
principal bundle with fiber the product of a Cousin group and possibly
$(\mbb{C}^*)^k$. Applying Lemma~\ref{Lem:Kiselman} we see that $G/\ol{H}$ is
pseudoconvex, thus holomorphically convex by Theorem~\ref{projcase}.
Since $\mathscr{O}(G/H)=\mbb{C}$, the space $G/\ol{H}$ is compact and thus homogeneous rational.    

In order to complete the proof we will show that the principal bundle
$G/H\to G/\ol{H}$ is holomorphically trivial and that $\ol{H}/H$ is a Cousin group   
and we show the latter first.  
Since $\ol{H}/H=C\times(\mbb{C}^*)^k$ where $C$ is a Cousin group, we have 
the factorization $G/H\to G/L\to G/\ol{H}$ where $L/H=C$ et $\ol{H}/L=(\mbb{C}^*)^k$.   
Due to Lemma~\ref{Lem:TrivialBundle} below, the principal bundle $G/L\to G/\ol{H}$ is
holomorphically trivial.    
Hence, $k\geq1$ would contradict the fact that every
holomorphic function on $G/H$ is constant.   
This shows that $\ol{H}/H$ is a Cousin group.

Finally, let us consider the two fibrations
\begin{equation*}
\xymatrix{
G/H \ar[r]^-{p_1}\ar[d]_{p_2} & G/\ol{H}\simeq G'/(G'\cap\ol{H})\\
G/G'H,
}
\end{equation*}
where $G'/(G'\cap\ol{H})$ is homogeneous rational and $G/G'H$ is a Cousin
group. The restriction of $p_1$ to the $p_2$-fiber $G'/(G'\cap H)$ is still
surjective. Thus it is a holomorphic bundle with fiber the algebraic variety
$(G'\cap \ol{H})/(G'\cap H)$ which is a closed subgroup of the Cousin group
$\ol{H}/H$. This implies that $(G'\cap\ol{H})/(G'\cap H)$ is finite. Since the
parabolic group $G'\cap\ol{H}$ is connected, we obtain $G'\cap\ol{H}= G'\cap
H$, and hence $G/H$ is the product of $G/\ol{H}$ and $\ol{H}/ H$.
\end{proof}

\begin{lem}\label{Lem:TrivialBundle}
Let $S$ be a connected semisimple complex Lie group, let $P$ be a parabolic
subgroup of $S$, and let $p\colon X\to S/P$ be an equivariant holomorphic
principal bundle with structure group $T=(\mbb{C}^*)^k$. If $X$ is pseudoconvex,
then the bundle is trivial.
\end{lem}

\begin{proof}
An equivariant holomorphic $T$--principal bundle $p\colon X\to S/P$ is of the
form $X\simeq S\times_PT$ where the twisted product is defined by a holomorphic
group homomorphism $P\to T$. Since $T$ is Abelian, this homomorphism
factorizes over $P/P'\simeq(\mbb{C}^*)^l$. Since $P/P'$ is reductive, the
factorized homomorphism is algebraic and in particular its image is an
algebraic subtorus $\wt{T}$ of $T$. Hence, we obtain a second fiber bundle
\begin{equation*}
\xymatrix{
X\simeq S\times_PT\ar[d]_p\ar[r] & T/\wt{T}\\
S/P.
}
\end{equation*}
Let us assume in a first step that $S={\rm{SL}}(2,\mbb{C})$ and thus that $P$ is
a Borel subgroup of $S$. Then $P/P'\simeq\mbb{C}^*$ and the bundle $X\simeq
S\times_PT\to S/P\simeq\mbb{P}_1$ is non-trivial if and only if $P/P'\to T$ is
non-constant. If this is the case, the fiber of $S\times_PT\to T/\mbb{C}^*$ is
isomorphic to a finite quotient of $S/P'\simeq\mbb{C}^2\setminus\{0\}$, hence we
find a closed embedding of such a finite quotient of $\mbb{C}^2\setminus\{0\}$
inside $X$. Since this implies that $X$ is not pseudoconvex, we have proved the
claim for $S={\rm{SL}}(2,\mbb{C})$.

For arbitrary $S$ we find root subgroups $S_\alpha\simeq{\rm{SL}}(2,\mbb{C})$ of
$S$ such that $S_\alpha\cap P$ is a Borel in $S_\alpha$. If the homomorphism
$P\to T$ is not trivial, then its restricition to $S_\alpha\cap P$ is not
trivial for some root $\alpha$. Since $S_\alpha\times_{(S_\alpha\cap P)}T$ is a
closed complex submanifold of $X$, this is in contradiction with the previous
case.
\end{proof}

The following example shows that a general pseudoconvex reductive homogeneous
manifold is \emph{not} a Cousin bundle over a holomorphically convex manifold.

\begin{ex}
Let $\Gamma\subset S={\rm{SL}}(2,\mbb{C})$ be a cocompact discrete subgroup
such that $\Gamma/\Gamma'$ contains an element of infinite order. Existence of
such $\Gamma$ is shown in e.g.~\cite{Mill}. Then there is a homomorphism
$\varphi\colon\Gamma\to\mbb{C}^*$ with dense image in $S^1\subset\mbb{C}^*$. We
define the reductive homogeneous manifold $X=G/\Gamma_G$ where
$G:=S\times\mbb{C}^*$ and $\Gamma_G$ is the graph of $\varphi$, hence a discrete
subgroup of $G$. By construction, $X$ is the total space of a holomorphic
$\mbb{C}^*$-principal bundle over the compact base $S/\Gamma$.

We claim that $X$ is pseudoconvex. Let $\rho$ be an $S^1$-invariant strictly
plurisubharmonic exhaustion of $\mbb{C}^*$. Then the function
$G\to\mbb{R}^{\geq0}$, $(s,z)\mapsto\rho(z)$, is a $\Gamma_G$-invariant
plurisubharmonic function of $G$, hence descends to a plurisubharmonic function
on $X=G/\Gamma_G$. The fibers of this function are the closures of finitely many 
$S$-orbits in $X$, thus compact. Therefore $X$ is indeed pseudoconvex.

One sees that $\wh{\Gamma}_G=S$ has no locally closed orbit in $X$ and that
$\mathscr{O}(X)=\mbb{C}$ so that $X$ is neither K\"ahler nor holomorphically
convex.
\end{ex}
 
\section{The structure of pseudoconvex solvmanifolds}

In this section we prove a structure theorem for pseudoconvex solvmanifolds,
i.e., homogeneous spaces $X=G/H$ where $G$ is connected solvable. Replacing $G$
by its universal covering we will assume from now on that $G$ is simply connected.   
Note that then every connected Lie subgroup of $G$ is automatically
closed and simply connected.

We start with the following observation from~\cite{Hu10} which deals with the
nilpotent case.

\begin{thm}\label{Thm:nilpotent}
Assume that $G$ is nilpotent and that $H$ is a closed complex subgroup of $G$.
Then the complex nilmanifold $X = G/H$ is pseudoconvex.
\end{thm} 

\begin{proof}
The normalizer $\mathscr{N}_G(H^0)$ in $G$ of the connected component
$H^0$ of the identity of $H$ is connected, hence also simply connected, e.g.,
see Lemma~2 in~\cite{Mat}. So $G/\mathscr{N}_G(H^0)$ is biholomorphic to
$\mathbb C^k$ for some $k$. By the Oka Principle the bundle $G/H \to
G/\mathscr{N}_G(H^0)$ is holomorphically trivial. Therefore we need only
consider its fiber which has the form $N/\Gamma$, where $\Gamma := H/H^0$ is a
discrete subgroup of the simply connected group $N := \mathscr{N}_G(H^0)/H^0$.
In the case of a connected, simply connected nilpotent Lie group the exponential
map $\exp\colon\mathfrak n \to N$ is biholomorphic. The pre-image of $\Gamma$ in
$\mathfrak n$ spans a (real) Lie subalgebra $\mathfrak n_{\Gamma}$ whose
associated connected Lie group $N_{\Gamma}$ contains $\Gamma$ cocompactly. We
now set $\wh{\mathfrak n}_{\Gamma} := \mathfrak n_{\Gamma} + i \mathfrak
n_{\Gamma}$ and let $\wh{N}_{\Gamma}$ denote the corresponding connected complex
Lie group. Then $N/\wh{N}_{\Gamma}$ is biholomorphic to $\mathbb C^l$ for some
$l$ and applying the Oka Principle again we see that $N/\Gamma$ is biholomorphic
to the product $\wh{N}_{\Gamma}/\Gamma \times \mathbb C^l$. So finally we see
that it suffices to consider the nilmanifold $\wh{N}_{\Gamma}/\Gamma$ in order
to ensure the existence of a pseudoconvex exhaustion on $X$.

Setting $\mathfrak m := \mathfrak n_{\Gamma} \cap i \mathfrak n_{\Gamma}$ and
letting $M$ denote the corresponding closed complex subgroup of
$\wh{N}_{\Gamma}$,
we consider the pair $(\wh{N}_{\Gamma}/M,N_{\Gamma}/M)$. Note that this pair is
pseudoconvex in the sense of Loeb \cite{Lo}, since nilpotent Lie algebras always
have purely imaginary spectra. So there exists an $(N_{\Gamma}/M)$-right
invariant strictly plurisubharmonic exhaustion function on $\wh{N}_{\Gamma}/M$
that pulls back to an $N_{\Gamma}$-right invariant plurisubharmonic exhaustion
function on $\wh{N}_{\Gamma}$. Thus we see that the nilmanifold
$\wh{N}_{\Gamma}/\Gamma$ is pseudoconvex. It then follows that the original
nilmanifold $G/H$ is also pseudoconvex.
\end{proof}

\begin{defn} 
A {\it (principal) Cousin group tower of length one} is a Cousin group. A {\it
(principal) Cousin group tower of length} $n > 1$ is a (principal) holomorphic
bundle with fiber a Cousin group and base a (principal) Cousin group tower of
length $n-1$.  
\end{defn}

The fiber of the holomorphic reduction of a nilmanifold carries no non-constant
holomorphic functions, in fact it is a Cousin group tower. For a general
solvmanifold $X=G/H$ this is no longer true as we have seen in
Example~\ref{Ex:CoeureLoeb} where the fiber of the holomorphic reduction is
$\mbb{C}^*\times\mbb{C}^*$. However, in the rest of this section we will prove
that pseudoconvexity of the solvmanifold $X=G/H$ is sufficient (though not
necessary, see Example~\ref{Ex:end}) to ensure that the fiber of its holomorphic
reduction is a Cousin group tower and therefore has no non-constant holomorphic
functions.

We will need the following observation.

\begin{lem}\label{Lem:closedness}
Let $G$ be a connected complex Lie group and $\Gamma\subset G$ a discrete
subgroup. Furthermore let $H_1 \subset H_2 \subset G$ be two closed connected
complex subgroups which are both normalized by $\Gamma$. Suppose that $\Gamma\!
H_2$ is closed in $G$ and that $(\Gamma \cap H_2)H_1$ is closed in $H_2$. Then
$\Gamma\! H_1$ is closed in $G$.
\end{lem}
  
\begin{proof}
Any subgroup of a not necessarily connected Lie group $L$ is closed in $L$ if and
only if its intersection with $L^0$ is closed in $L^0$. Therefore $\Gamma\!H_1$
is closed in $\Gamma\!H_2$ (and hence closed in $G$) if and only if
$\Gamma\!H_1\cap H_2$ is closed in $H_2$. Since $\Gamma\!H_1\cap H_2=(\Gamma\cap
H_2)H_1$, the lemma is proved.
\end{proof}

We prove our result first for a special class of solvmanifolds.

\begin{prop}\label{solvpscx}   
Let $X=G/\Gamma$ be a solvmanifold with $\Gamma\subset G$ discrete such that
$\mathscr{O}(X) \simeq \mathbb C$. If $X$ is pseudoconvex, then there is a
connected normal Abelian subgroup $C \subset G$ such that $\Gamma C$ is closed
in $G$ and the fibration $G/\Gamma \to G/\Gamma C$ has the Cousin group $C
\Gamma/ \Gamma$ as fiber.
\end{prop}

\begin{proof}
Since $G'$ is a connected nilpotent subgroup of $G$, the exponential map
$\exp\colon\lie{g}'\to G'$ is biholomorphic and we can construct the smallest
connected complex subgroup of $G'$ containing $\Gamma'$ as in the proof of
Theorem~\ref{Thm:nilpotent}, i.e., by $S_1:= \wh{G}'_{\Gamma'}$. Let $L\subset
G$ be the identity component of the centralizer of $S_1$ in $G$. As remarked in
the proof of~\cite[Abspaltungssatz~3.2]{BaOt69}, the group $L \Gamma$ is a
closed subgroup of $G$. Therefore, it follows from~\cite[Satz~1.1]{BaOt69} that
$L$ is normal in $G$. One gets the fibration 
\[
X=G/\Gamma \to G/L\Gamma.
\]
By hypothesis, $X$ is a pseudoconvex manifold and therefore there is a maximal
connected complex subgroup $\wh{H} \subset G$ such that $\wh{H}$ is normalized
by $\Gamma$ and $\wh{H} \Gamma/\Gamma$ is relatively compact in $X$. Since
$\mathscr{O}(X) \simeq \mathbb C$, we have $\wh{H} \lhd G$. Consider the Lie
algebra $[\g,\wh{\lie{h}}] \subset \g'$. If  $[\g,\wh{\lie{h}}] = 0$, then
$\wh{\lie{h}} \subset \mathfrak z(\g)$ and hence $\wh{\lie{h}} \subset \lie{l}$.
If $[\g,\wh{\lie{h}}] \neq 0$, then $0 \neq \wh{\lie{h}} \cap \g'\lhd \g $ and in
particular $\wh{\lie{h}}\cap\g' \lhd \g'$. The Lie algebra $\wh{\lie{h}}\cap\g'$
being normal in $\g'$, it follows that $(\wh{\lie{h}}\cap\g') \cap \mathfrak
z(\g') \neq 0$. Since $\mathfrak z(\g') \subset\lie{l}$, it follows that in both
cases $\wh{\lie{h}}_{\lie{l}} := \wh{\lie{h}}\cap\lie{l} \neq 0$ is a non-trivial
ideal in $\g$.

Let $\tilde L \subset L$ be the smallest connected subgroup containing
$\wh{H}_L:= \wh{H}\cap L$ such that $\tilde L$ is normalized by $\Gamma$,
$\tilde L \Gamma$ is closed in $G$ and $\mathscr{O}(\tilde L \Gamma/\Gamma)
\simeq \mathbb C$. Then $\tilde L$ again is a normal subgroup of $G$. In order
to simplify the notation, we omit the tilde and now have the fibration    
\[   
        X=G/\Gamma \to  G/L \Gamma     
\]  
with $\mathscr{O}(L \Gamma\!/\Gamma) \simeq \mathbb C$.   Note that the fiber is
isomorphic to $L/(L\cap \Gamma)$ where $L$ is a simply connected solvable Lie
group. Now define $\La:=\Gamma \cap L$. Since $L\cap \wh{H}$ is non-trivial, the
discrete group $\La$ is also non-trivial and a normal subgroup of $\Gamma$. We
will now apply the same reduction steps as before to the homogeneous manifold
$L/\Lambda$ which has the same properties as $X=G/\Gamma$. In order to be able
to carry over the results of this iteration to $L\Gamma/\Gamma$ we must check
carefully that all the groups constructed inside $L$ are normalized by $\Gamma$.

If $\La'$ is trivial, then $\La$ and hence $L$ are abelian and the lemma is
proved. Therefore suppose that $S_2:=\wh{L}'_{\Lambda'}\subset L'$ is
non-trivial and let $L_2$ be the connected component of the centralizer of $S_2$
in $L$. Since $\La \lhd \G$, we have that $\G$ normalizes $S_2$ and hence   
$L_2$. As a consequence we get that $L_2 \lhd G$. Furthermore
$\wh{\lie{h}}_{\lie{l}} \cap \lie{l}_2\neq 0$, for the same reasons as above. By
Lemma~\ref{Lem:closedness} we get that $\G\! L_2$ is a closed subgroup of $G$.

Now we can iterate the construction from the proof
of~\cite[Abspaltungssatz~3.2]{BaOt69} to produce a chain of subgroups $L\supset
L_2\supset\dotsb$ in $G$ which are normalized by $\Gamma$ and such that $\Gamma
L_j$ is closed in $G$. Since this process must terminate after finitely many
steps, we finally get the Abelian group $C$ such that
$\mathscr{O}(C\Gamma/\Gamma)\simeq\mbb{C}$ as claimed.
\end{proof}  

After these preparations we are able to prove the main result of this section.

\begin{thm}\label{Thm:solv}
Suppose $X=G/H$ is a pseudoconvex solvmanifold with holomorphic reduction $G/H
\to G/J$. Then the base $G/J$ is Stein and the fiber $J/H$ is a Cousin group
tower. In particular, $\mathscr{O}(J/H) \simeq \mathbb C$.
\end{thm}

\begin{proof}
The fact that $G/J$ is Stein is proven in~\cite{HuOe86}. We have $\wh{\lie{h}}
\subset\lie{j}$ since the base is holomorphically separable by definition. We
claim that $\mathscr{O}(J/H)\simeq\mathbb C$. If not, then there would exist a
holomorphic reduction $J/H \to J/J_1$ of $J/H$ with $\dim J > \dim J_1$. We would
continue taking holomorphic reductions in this way: given $J_{n-1}$ we define the
closed complex subgroup $J_n$ by means of the holomorphic reduction $J_{n-1}/H\to
J_{n-1}/J_n$, where $\dim J_n < \dim J_{n-1}$. The process stops after a finite
number of steps and we obtain a subgroup $J_k$ with
$\mathscr{O}(J_k/H)\simeq\mathbb C$, where we assume that $k$ is the smallest
positive integer such that $J_k/H$ has this property. Note that $\dim J_k/H >
0$, since recursively we have $\wh{\lie{h}}\subset \lie{j}_n$ for $1 \le n \le
k$ and $\dim \wh{\lie{h}} > \dim \lie{h}$.

We claim that $J_k/H$ is a Cousin group tower. Since
$J_k/\mathscr{N}_{J_k}(H^0)$ is Stein by Lie's Flag Theorem, it follows that
$\mathscr{N}_{J_k}(H^0) = J_k$ and thus the isotropy is discrete. In this case
we write $\Gamma$ instead of $H$. We now apply Proposition~\ref{solvpscx} to
$J_k/\Gamma$ and get the requisite subgroup $C\subset J_k$ such that $J_k/\Gamma
\to J_k/C\Gamma$ has a Cousin group as fiber. By Lemma~\ref{Lem:Kiselman} the
base $J_k/C\Gamma$ is pseudoconvex and $\mathscr{O}(J_k/C\Gamma) \simeq \mathbb
C$. By recursion, $J_k/\Gamma$ is a Cousin group tower and so by repeated use of
Lemma~\ref{Lem:Kiselman} we conclude that $G/J_k$ is pseudoconvex.

Finally, we shall show that $k\ge 1$ yields a contradiction. Consider the
subgroups $H\subset J_k \subset J_{k-1}\subset J_{k-2}\subset G$, where we set
$J_{-1}:=G$ and $J_{0}:=J$. Then $J_{k-2}/J_{k}$, as a closed submanifold of
$G/J_{k}$, is pseudoconvex. Moreover, we have the fibration $J_{k-2}/J_{k} \to
J_{k-2}/J_{k-1}$. Both $J_{k-2}/J_{k-1}$ and $J_{k-1}/J_k$ are holomorphically
separable implying $J_{k-2}/J_k$ is Stein by Remark \ref{Serre}. This implies
$J_{k-2}/J_k$ is the holomorphic reduction of $J_{k-2}/H$. But, by construction
$J_{k-2}/J_{k-1}$ is the holomorphic reduction of $J_{k-2}/H$. Since $\dim J_k <
\dim J_{k-1}$,  we obtain the desired contradiction. \emph{A posteriori} we see
that $J=J_k$, i.e., that $\mathscr{O}(J/H)\simeq\mathbb C$. The argument given
in the previous paragraph then shows that $J/H$ is indeed a Cousin group tower.
\end{proof}

We finish with an example which shows that the converse of
Theorem~\ref{Thm:solv} does not hold.

\begin{ex}\label{Ex:end}
Let $n \geq 3$ and $A \in {\rm SL}(n, \mathbb Z)$ such that $A$ is
diagonalizable over $\mathbb C$ and admits $s>0$ positive real eigenvalues
$\alpha_1,...,\alpha_s$ and $t>0$ pairs of complex conjugate eigenvalues
$\beta_1,\bar{\beta_1},...,\beta_t,\bar{\beta_t}$. Note that $n=s+2t$. Assume
furthermore that the characteristic polynomial of $A$ is irreducible over
$\mbb Q$. In particular $ \alpha_j \neq 1$ for all $j$. The existence of such an
$A$ is easily seen by using elementary number theory, see e.g.~\cite{OT}. The
fact that $A$ is diagonalizable implies that there is a \emph{real} logarithm $D
\in \mathfrak{sl}(n,\mathbb R)$ of $A$. This means that the one-parameter group 
$\{\exp(x D)\! \mid \! x \in \mbb R\}$ lies in ${\rm SL}(n, \mathbb R)$, which
justifies the following construction. For $K=\mbb Z, \mbb R, \mbb C$ define a
solvable group structure $G_K:=K \ltimes K^n$ on the cartesian product by
\begin{equation*}
(x_1,v_1)\cdot (x_2,v_2):=\bigl(x_1+x_2,\exp(x_1D)v_2+v_1\bigr).
\end{equation*}
Note that $G_{\mbb Z}$ is discrete cocompact in $G_{\mbb R}$ and that $G_{\mbb
R}$ is a real form of the simply-connected solvable complex Lie group $G_{\mbb
C}$. Now let $H \subset \{0\} \ltimes \mbb C^n \subset G_{\mbb C}$ be the 
$t$-dimensional complex subgroup generated by $A$-eigenvectors corresponding to
the eigenvalues $\beta_1,\beta_2,\dotsc,\beta_t$. It was shown in~\cite{OT} that
$C:=\{0\} \ltimes\mbb C^n/(\{0\} \ltimes \mbb Z^n) H$ is a Cousin group. Note
that $H$ is a normal subgroup of $G_{\mbb C}$. Now define $X:= G_{\mbb
C}/G_{\mbb Z} H$. We get a (non-principal) fibration $X:= G_{\mbb C}/G_{\mbb Z}
H \to G_{\mbb C}/G_{\mbb Z} (\{0\} \ltimes \mbb C^n)$ with  the Cousin group $C$
as fiber and $\mbb C^*$ as base.    

Suppose that $X$ is pseudoconvex and choose a
plurisubharmonic exhaustion function. Since $G_{\mbb R} H/G_{\mbb Z} H$ is
compact, an integration argument as in \cite{Lo} gives a plurisubharmonic
function on $G_{\mbb C}$ which is $G_{\mbb R} H$-invariant and is an exhaustion
function on $G_{\mbb C}/G_{\mbb R} H$. The maximal connected complex subgroup of
$G_{\mbb R} H$ being $H \oplus \ol{H}$, we get the pseudoconvex couple $(G_{\mbb
C}/H \oplus \ol{H},G_{\mbb R}/H \oplus \ol{H})$ in the sense of Loeb. But this
couple has eigenvalues $\ln(\alpha_j)\in \mbb R^*$ of the adjoint representation
of its real form on itself. This is a contradiction. Therefore $X$ is a Cousin
fiber bundle over a Stein manifold but is \emph{not} pseudoconvex.
\end{ex}

\end{document}